\documentclass{amsart}
\usepackage{amsfonts}
\usepackage{amsmath,amscd}
\usepackage{amssymb}
\usepackage{amsthm}
\usepackage{newlfont}

 \newtheorem{thm}{Theorem}[section]
\newtheorem{cor}[thm]{Corollary}
 \newtheorem{lem}[thm]{Lemma}
 \newtheorem{prop}[thm]{Proposition}
\theoremstyle{definition}
 
\theoremstyle{remark}

 \theoremstyle{problem}
 
 \numberwithin{equation}{section}

\begin{document}

\title[A note on some special $p$-groups]
 {A note on some special $p$-groups}

\author[F. Johari]{Farangis Johari}
\author[P. Niroomand]{Peyman Niroomand}

\email{e-mail:farangis.johari@mail.um.ac.ir,farangisjohary@yahoo.com}
\address{Department of Pure Mathematics\\
Ferdowsi University of Mashhad, Mashhad, Iran}

\address{School of Mathematics and Computer Science\\
Damghan University, Damghan, Iran}
\email{niroomand@du.ac.ir, p$\_$niroomand@yahoo.com}



\date{\today}


\begin{abstract}
Recently Rai obtained  an upper bound for the order of the Schur multiplier of a $d$-generator   special $p$-group when its derived subgroup has the maximum value  $ p^{\frac{1}{2}d(d-1)}$ for $ d\geq 3 $ and $ p\neq 2. $ Here we try to obtain
the Schur multiplier, the exterior square and the tensor square of  such  $p$-groups.  Then we specify which ones are capable. Moreover, we give an upper bound for the order of the Schur multiplier, the exterior product and the tensor square of a $d$-generator   special $p$-group $ G $ when $ |G'|=p^{\frac{1}{2}d(d-1)-1}$ for $ d\geq 3 $ and $ p\neq 2. $ Additionally, when $ G $ is of exponent $ p, $ we give the structure of $ G. $
\end{abstract}

\maketitle

\section{Motivation and Preliminaries}
The Schur multiplier, $ \mathcal{M}(G), $ of a group $ G $
is introduced by I. Schur in his work on the projective representations of groups. Many authors tried to find the structure of the Schur multiplier for some class of groups. For instance, the  Schur  multiplier of abelian groups and extra-special
 $p$-groups are well-known (see \cite{kar}). The extra-special  $p$-group has the minimum value for the order of the  derived subgroup. Recently, Rai in \cite{raif}, obtained the structure of a $d$-generator special $p$-group  when the order of the derived subgroup has the maximum possible value
 $p^{ \frac{1}{2} d(d-1)}. $ He could obtain the Schur multiplier of such groups. \\
 Remember from \cite{br1}, $ G\otimes G $ is used to denote the non-abelian tensor product of $ G. $ The non-abelian exterior product $ G\wedge G $ of a group $ G $ is defined as $ G\otimes G/ \bigtriangledown (G),$ where $ \bigtriangledown (G)=\langle g\otimes g|g\in G\rangle. $
 We denote the image $ g\otimes g'  $ in $ G\otimes G/ \bigtriangledown (G)$ by $ g\wedge g' $ for all $ g,g'\in G. $ There is the homomorphism  $ \kappa: G\otimes G\rightarrow G'$ given by $ \kappa(g\otimes h)=[g,h]. $ Since $ \kappa $ has $ \bigtriangledown (G) $ in its kernel, the homomorphism  $ \kappa': G\wedge G\rightarrow G'$ is  induced by $\kappa.$ It is known that the  kernel of $ \kappa' $ is isomorphic to the Schur multiplier of $ G $
(see Lemma \ref{j1}). 

 In the current paper, we obtain the non-abelian tensor square and the exterior product of a  special $p$-group $  G $ of rank $ \frac{1}{2}d(d-1) $ $ (|G'|=|Z(G)|
 = |\Phi(G)|=p^{\frac{1}{2}d(d-1)}). $ Moreover, we show which of them are capable.
Recall that a group $ G $ is called capable  provided that $ G\cong H/Z(H) $ for a group $ H. $\\
In this paper, in the same motivation to \cite{raif}, we are interesting to obtain the structure of
$d$-generator special $p$-groups of exponent $ p~ (p\neq 2)$ when $ |G'|=p^{\frac{1}{2}d(d-1)-1}. $  Then we obtain the structure of the Schur multiplier, the exterior product, the tensor square and
third homology group $ \pi_3(SK(G,1)) $  of suspension of an Eilenberg-MacLane space $ K(G,1) $ (the kernel of $ \kappa $) when $ G $ belongs to this class of groups.\\
 Beyl et al. in \cite{3} gave a
criterion for detecting capable groups. They showed  a group $G$ is capable if and only if the epicenter of $G,Z^*(G) ,$ is trivial.
Ellis in \cite{111} showed $Z^{\wedge}(G)=Z^*(G),$  where the exterior center $Z^{\wedge}(G)$ of $G$ is the set of all elements $g$ of $G$ for which
$g \wedge h = 1_{G\wedge G}$ for all $h \in  G$
(see for instance \cite{111}  to find  more information in  this topics).

The next lemma gives a criterion to detecting the capable $p$-groups.

\begin{lem}\cite[Corollary 4.3]{3}\label{a} A group $G$ is capable if and only if the natural map $\mathcal{M}(G)\rightarrow  \mathcal{M}(G/\langle x\rangle)$ has a non-trivial kernel for all non-zero elements $x \in Z(G).$
\end{lem}
The next result shows the kernel of commutator map $ \kappa' $ is isomorphic to the Schur multiplier.

\begin{lem}\cite{br1,br2}\label{j1}
Let $ G$ be a group. Then
$1\rightarrow \mathcal{M}(G)\rightarrow G\wedge G \xrightarrow{\kappa'} G'\rightarrow 1$ is exact.

\end{lem}
The next result is extracted from the work of Blackburn and Evens in
\cite[Remark, Section 3]{black burn} and \cite[Corollary 3.2.4]{kar}.
Let $ F/R $ be a free presentation for a group $ G $ and $ \pi:\tilde{x} R\in F/R \mapsto x\in G $  be a natural homomorphism. Then
\begin{prop}\label{lklklk}
Let $ G $ be a finite  non-abelian $p$-group of class two such that $ G/G' $ is elementary abelian. Then  \[1\rightarrow \ker \beta \rightarrow G'\otimes G/G' \xrightarrow{\beta} \mathcal{M}(G)\rightarrow
\mathcal{M}(G/G')\rightarrow G' \rightarrow 1\]
is exact, in where
\[\beta: x\otimes (zG')\in G'\otimes G/G'\rightarrow [\tilde{x},\tilde{z}][R,F]\in \mathcal{M}(G)=(R\cap F')/[R,F],\]
 $ \pi(\tilde{x}R)=x $ and $ \pi(\tilde{z}R)=z.$
Moreover,
$ \ker \beta=\langle ([x,y]\otimes zG')
([z,x]\otimes y G')([y,z]\otimes xG'),w^p\otimes wG'|x,y,z,w\in G\rangle .$
\end{prop}
\begin{thm}(\cite[Proposition 1]{ele} and \cite{el})\label{lkk}
Let $ G $ be a finite  non-abelian $p$-group of class $ c. $
The map \[\Psi_2:  G/G'Z(G)\otimes  G/G'Z(G) \otimes G/G'Z(G)  \rightarrow \big{(}G'/\gamma_3(G)\big{)}\otimes G/G'\] given by
$xG'Z(G)\otimes yG'Z(G) \otimes zG'Z(G)\mapsto $\[ ([x,y]\gamma_3(G)\otimes zG' )+
([z,x]\gamma_3(G)\otimes yG')+([y,z]\gamma_3(G)\otimes xG')
\]
is a  homomorphism. If any two elements of the set $\{ x,y,z\} $ are linearly dependent. Then $\Psi_2(xG'Z(G) \otimes yG'Z(G) \otimes zG'Z(G))=1_{\big{(}G'/\gamma_3(G)\big{)}\otimes G/G'}.$
\end{thm}

The following result gives a bound for the minimal generating set of $ G' $ for a group $ G $ of the nilpotency class $ 2. $
\begin{lem}\label{ll}\cite[Lemma 1.11]{nif}
Let $ G $ be a group of the nilpotency class two such that $ d(G/Z(G))=d $ is finite. Then $ d(G')\leq \frac{1}{2} d(d-1).$
\end{lem}

\begin{lem}\label{lclass}
Let $ G $ be a group of the nilpotency class two. Then
\[( [x,y]\wedge z)([z,y]\wedge x)(x\wedge [z,y])=1_{G\wedge G} \] for all $ x,y,z\in G. $
\end{lem}
\begin{proof}
Since $ G'\subseteq Z(G), $ we have
\begin{align*}
&[x,y]\wedge z=(x\wedge y) ^{z}(x\wedge y)^{-1}=(x\wedge y)([z,x]x\wedge [z,y]y)^{-1}\\&=(x\wedge y)([z,x]\wedge y)^{-1} (x\wedge [z,y])^{-1}(x\wedge y)^{-1}=([z,x]\wedge y)^{-1} (x\wedge [z,y])^{-1}.
\end{align*}
Thus  \[( [x,y]\wedge z)([z,y]\wedge x)(x\wedge [z,y])=1_{G\wedge G}. \]
\end{proof}
Let
 $ \mathbb{Z}_{p^{\alpha}}^{(r)} $ denote the direct product of $ r$-copies of $ \mathbb{Z}_{p^{\alpha}}. $
\begin{lem}\cite[Corollary 2.2.12]{kar}\label{11}
Let $G\cong \mathbb{Z}_{p^{m_1}}\oplus \cdots\oplus \mathbb{Z}_{p^{m_k}},$ where $m_1 \geq \ldots \geq m_k$. Then
\[
\mathcal{M}(G)\cong  \bigoplus_{i=2}^{k}\mathbb{Z}^{(i-1)}_{p^{m_i}}.
\]
\end{lem}

We need to recall that the concept of basic commutators.
Let $X$ be an arbitrary subset of a free group, and select an arbitrary total order for $X$. The basic commutators on $X$, their weight $wt$, and the order among them are defined as follows:
\begin{itemize}
\item[(i)]The elements of $X$ are basic commutators of weight one, ordered according to the total order previously chosen.
\item[(ii)]Having defined the basic commutators of weight less than $n$, a basic commutator of weight $n$ is $c_k=[c_i,c_j]$, where:
\begin{itemize}
\item[(a).]$c_i$ and $c_j$ are basic commutators and $wt(c_i)+wt(c_j)=n$, and
\item[(b).]$c_i>c_j$, and if $c_i=[c_s,c_t]$, then $c_j \geq c_t$.
\end{itemize}
\item[(iii)] The basic commutators of weight $n$ follow those of weight less than $n$. The basic commutators of weight $n$ are ordered among themselves in any total order, but the most common used total order is lexicographic order; that is, if $[b_1,a_1]$ and $[b_2,a_2]$ are basic commutators of weight $n$. Then $[b_1,a_1]<[b_2,a_2]$ if and only if  $b_1<b_2$ or $b_1=b_2$ and $a_1<a_2$.
\end{itemize}

The number of basic commutators is given in the following:

\begin{thm}\label{jkhhhhh}
(Witt Formula) The number of basic commutators of weight $n$ on $d$
generators is given by the following formula:
$$ \chi_n(d)=\frac {1}{n} \sum_{m|n}^{} \mu (m)d^{n/m},$$
where $\mu (m)$ is the M\"{o}bius function, which is defined to be
   \[ \mu (m)=\left \{ \begin{array}{ll}
      1 & \textrm{if}\ \ m=1, \\ 0 &  \textrm{if} \ \ m=p_1^{\alpha_1}\ldots
p_k^{\alpha_k}\ \ \exists \alpha_i>1, \\ (-1)^s & \textrm{if} \ \
m=p_1\ldots p_s,
\end{array} \right.  \] where the $p_i$, $1\leq i\leq k$, are the distinct primes dividing
$m.$
\end{thm}

\begin{thm}\label{hh}
$($ \cite[M. Hall]{m} and \cite[Theorem 11.15 (a)]{Hu} $)$
Let $F$ be a free group on $\{x_1, x_2,\ldots, x_d\}.$ Then for all $1 \leq i \leq n$, \[\dfrac{\gamma_n(F)}{\gamma_{n+i}(F )}\] is a free abelian group freely generated by the basic commutators of weights $n, n + 1, \ldots, n + i - 1$ on the letters $\{x_1, x_2, \ldots, x_d\}.$ 
\end{thm}

\section{Some results on the capability and Schur multiplier of $d$-generator   special $p$-groups of rank $\frac{1}{2}d(d-1)$ or $\frac{1}{2}d(d-1)-1$}
In this section, as application to the result of Rai in \cite{raif}, in the class of all $d$-generator  special $p$-group of rank  $ \frac{1}{2}d(d-1)$ for $ d\geq 3 $ and $ p\neq	2, $  we specify which ones are capable. Among the others results, we give an upper bound for the Schur multiplier of $d$-generator   special $p$-groups of rank $\frac{1}{2}d(d-1)-1.$ Moreover, we give the structure of $d$-generator  special $p$-group of rank  $ \frac{1}{2}d(d-1)-1$  of exponent $ p $ for $ p\neq	2 $ and $  d\geq 3 $ and we show they attain the mentioned upper bound.\\
Let $ F $ be a free group. Then $G= F/F^p \gamma_3(F)$ is the relatively  free group in the class of nilpotent group of class two and  exponent $ p.$ From \cite[Section $ 5 $]{roz}, it is well-known that $ G\cong \langle x_1,\ldots,x_d|x_t^p=[x_i,x_j]^p=[x_i,x_j,x_t]=1,1\leq i,j,t\leq d\rangle $ and so $ G'\cong \mathbb{Z}^{(\frac{1}{2}d(d-1))}_p. $
Let $ exp(X) $ be used to denote the exponent of $ X. $ Then
\begin{lem}\label{h1}
 Let $ H $ be a  finite special $p$-group  such that $ p\neq	2 $ and $ d(H)= d\geq 3 .$ Then    $|H'|=p^{ \frac{1}{2}d(d-1)}$ and   $exp(H)= p $ if and only if $ |H|=p^{\frac{1}{2}d(d+1) },$  $ H$ is a free relatively   special $p$-group of rank  $ \frac{1}{2}d(d-1)$ of exponent $ p $ and $H\cong \langle x_1,\ldots,x_d|x_t^p=[x_i,x_j]^p=[x_r,x_s,x_t]=1,1\leq i<j\leq d, 1\leq r,s,t\leq d\rangle. $
\end{lem}
\begin{proof}
Let  $|H'|=p^{ \frac{1}{2}d(d-1)}$ and   $exp(H)= p. $
We can choose a generating set $ \{ x_1Z(H),\ldots,x_dZ(H)\} $ for $ H/Z(H)$ such that $ [x_i,x_j] $ is non-trivial for $ i\neq j. $
It is clear   to see that $ \{[x_i,x_j]|1\leq i<j\leq d\} $ is a generating set of $ H'.$ Since $Z(H)= H'\cong \mathbb{Z}^{(\frac{1}{2}d(d-1))}_p$ and $ exp(H)=p, $ we have
$ \langle x_1,\ldots,x_d|x_t^p=[x_i,x_j]^p=[x_r,x_s,x_t]=1,1\leq i<j\leq d, 1\leq r,s, t\leq d\rangle,$ by using \cite[Proposition 5.1]{roz}. The converse is clear.
\end{proof}
Let $ \beta $ be the homomorphism mentioned in the Proposition \ref{lklklk}. Then
 \begin{thm}\label{h2}
 Let $ H $ be a  $d$-generator finite special $p$-group of rank  $ \frac{1}{2}d(d-1),$ $ d\geq 3 $ and $ p\neq	2. $ Then
 \begin{itemize}
  \item[$(i)$] $\mathcal{M}(H)\cong \mathbb{Z}^{(\frac{1}{3}d(d-1)(d+1)-m')}_p $ and
 $ \ker \beta \cong \mathbb{Z}^{(\frac{1}{6}d(d-1)(d-2)+m')}_p, $  where $ m'=\log_{p}\dfrac{|\langle w^p\otimes wG'|w\in G\rangle |}{|\mathrm{Im}\Psi_2\cap \langle w^p\otimes wG'|w  \in G\rangle |}$ and $|\mathrm{Im}\Psi_2|=p^{\frac{1}{6}d(d-1)(d-2)}.  $
  \item[$(ii)$]
 $ \mathcal{M}(H)\cong \mathbb{Z}^{(\frac{1}{3}d(d-1)(d+1))}_p $ if and only if $ exp(H)=p $ and
$H\cong \langle x_1,\ldots,x_d|x_t^p=[x_i,x_j]^p=1,[x_i,x_j,x_t]=1,1\leq i,j,t\leq d\rangle. $
\end{itemize}
\end{thm}
\begin{proof}
The result follows from \cite[Theorem 1.1 ]{raif}, Proposition \ref{lklklk}, Theorem \ref{lkk} and Lemma \ref{h1}.
\end{proof}
\begin{prop}\label{hklkl5}
Let $ G $ be a $d$-generator special $p$-group of rank $ \frac{1}{2}d(d-1)$ or  $ \frac{1}{2}d(d-1)-1$ for $ p\neq 2,  d\geq 3 $ and $ |G^p|=p^t. $ Then $|  \langle w^p\otimes wG'|w\in G\rangle|=p^{\frac{1}{2}t(2d-t+1)}. $ 
\end{prop}
\begin{proof}
The result follows from \cite[Proposition 3.3]{raif}.
\end{proof}
\begin{prop}\label{hklkl}
Let $ G $ be a $d$-generator special $p$-group of rank $ \frac{1}{2}d(d-1)-1$ and  $d\geq 3.$ Then $B= \{ \Psi_2(x_nG'\otimes x_q G'\otimes x_k G')|1\leq n<q<k \leq d \} $ is a basis of $ \mathrm{Im}\Psi_2 $ and $|\mathrm{Im}\Psi_2|=p^{\frac{1}{6}d(d-1)(d-2)}.$
\end{prop}
\begin{proof}
Let $\{ x_1, x_2,\ldots, x_d\}$ be a minimal generating set of $G.$ Since $  G/G' $ and $ G' $ are elementary abelian $p$-groups, they can be considered as vector spaces over the field $\mathbb{Z}_p.$ Therefore $ G'\otimes G/G' $ can be considered as a vector space with a subspace $ \mathrm{Im}\Psi_2. $ Note that $ G'=\langle [x_{n},x_{q}]|1\leq n<q\leq d\rangle. $ Since $ G' $ is elementary abelian of rank $ \frac{1}{2}d(d-1)-1,$ we have $ [x_i,x_j]=\sum_{\overset{1\leq n<q\leq d}{(n,q)\neq (i, j) }}\beta_{nq}[x_{n},x_{q}], $ where $0\leq  \beta_{nq} < p,$ for some  $1\leq i<j \leq d.$   We know that $ B_1=\{ [x_r,x_t]|1\leq r<t \leq d,(r,t)\neq (i, j)  \} $ generates $ G' $ and $ G/G' $ is generated by $ \{x_1G',\ldots,x_d G'\}. $
Since $G'\cong \oplus_{1\leq r<t \leq d,(r,t)\neq (i, j)} \langle[x_r,x_t] \rangle  $,   $ B_1 $ is also linearly independent.
Since \[ G'\otimes G/G'\cong \oplus_{1\leq k\leq d} \quad \oplus_{1\leq r<t \leq d,(r,t)\neq (i, j)} \langle [x_r,x_t]\otimes x_k G'\rangle, \] $A= \{[x_r,x_t]\otimes x_k G' |1\leq k\leq d,1\leq r<t \leq d,(r,t)\neq (i, j) \} $ is a basis of $ G'\otimes G/G'. $
We claim that $ \Psi_2(x_iG'\otimes x_j G'\otimes x_kG') $ is non-trivial, for all $1\leq i<j<k \leq d.$
By contrary, let $ \Psi_2(x_iG'\otimes x_j G'\otimes x_kG')=0. $ Then $[x_i,x_j]\otimes x_k G'+[x_j,x_k]\otimes x_iG' +[x_k,x_i]\otimes x_jG'=0$ so $ [x_k,x_i]\otimes x_jG'=-[x_i,x_j]\otimes x_k G'-[x_j,x_k]\otimes x_iG' $ which is contradiction since $[x_k,x_i]\otimes x_jG'\in A.$
%
We claim that  $B= \{ \Psi_2(x_nG'\otimes x_q G'\otimes x_k G')|1\leq n<q<k\leq d \} $  is a basis of $\mathrm{Im}\Psi_2.  $  Clearly, $ B $ generates $\mathrm{Im}\Psi_2.  $  Now, we show that $ B $ is linearly independent. Assume that $  \sum_{1\leq n<q<k \leq d} \alpha_{nqk} \Psi_2(x_nG'\otimes x_qG'\otimes x_kG')=0. $ Then $\sum_{1\leq n<q<k \leq d} \alpha_{nqk}([x_n,x_q]\otimes x_k G'+[x_q,x_k]\otimes x_nG' +[x_k,x_n]\otimes x_qG')=0.  $
 Hence
\begin{align*}
&\sum_{1\leq n<q<k \leq d} \big{(}\alpha_{nqk}[x_n,x_q]\otimes x_k G'+\alpha_{nqk}[x_q,x_k]\otimes x_nG' + \alpha_{nqk}[x_k,x_n]\otimes x_qG'\big{)}=
\\&\sum_{\overset{1\leq n<q<k \leq d}{(n,q)\neq (i,j) }}\big{(}\alpha_{nqk}[x_n,x_q]\otimes x_k G'+\alpha_{nqk}[x_q,x_k]\otimes x_nG' +\alpha_{nqk}[x_k,x_n]\otimes x_qG'\big{)}\\&+\sum_{i <j<k,1\leq k \leq d} \big{(}\alpha_{ijk}[x_i,x_j]\otimes x_{k} G'+\alpha_{ij k}[x_j,x_{k}]\otimes x_iG' +\alpha_{ij k}[x_{k},x_i]\otimes x_j G'\big{)}\\
&=
\sum_{\overset{1\leq n<q<k \leq d}{(n,q)\neq (i,j),j<k}} \big{(}\alpha_{nqk}[x_n,x_q]\otimes x_k G'+\alpha_{nqk}[x_q,x_k]\otimes x_nG' +\alpha_{nqk}[x_k,x_n]\otimes x_qG'\big{)}\\&+
\sum_{\overset{1\leq n<q<k \leq d}{(n,q)\neq (i,j),j\geq k}} \big{(}\alpha_{nqk}[x_n,x_q]\otimes x_k G'+\alpha_{nqk}[x_q,x_k]\otimes x_nG' +\alpha_{nqk}[x_k,x_n]\otimes x_qG'\big{)}
\\&
+\sum_{1\leq {k} \leq d,i<j<k} \big{(}\alpha_{ijk}[x_i,x_j]\otimes x_{k} G'+\alpha_{ij k}[x_j,x_{k}]\otimes x_iG' +\alpha_{ij k}[x_{k},x_i]\otimes x_j G'\big{)}=
0.\end{align*}
Since  $[x_i,x_j]=\sum_{\overset{1\leq n<q\leq d}{(n,q)\neq (i, j) }}\beta_{nq}[x_{n},x_{q}], $ where $0\leq  \beta_{nq} < p$ and $1\leq i<j \leq d,$ we have 
\begin{align*}
&\sum_{1\leq n<q<k \leq d} \big{(}\alpha_{nqk}[x_n,x_q]\otimes x_k G'+\alpha_{nqk}[x_q,x_k]\otimes x_nG' +\alpha_{nqk}[x_k,x_n]\otimes x_qG'\big{)}=\\&
\sum_{\overset{1\leq n<q<k \leq d}{(n,q)\neq (i,j),j<k}} \big{(}\alpha_{nqk}[x_n,x_q]\otimes x_k G'+\alpha_{nqk}[x_q,x_k]\otimes x_nG' +\alpha_{nqk}[x_k,x_n]\otimes x_qG'\big{)}+\\&
\sum_{\overset{1\leq n<q<k \leq d}{(n,q)\neq (i,j),j\geq k} }\big{(}\alpha_{nqk}[x_n,x_q]\otimes x_k G'+\alpha_{nqk}[x_q,x_k]\otimes x_nG' +\alpha_{nqk}[x_k,x_n]\otimes x_qG'\big{)}\\&
 +\sum_{1\leq k \leq d,i<j<k} \Big{(}
\big{(} \alpha_{ijk}
\sum_{\overset{1\leq n<q\leq d}{(n,q)\neq (i, j)}}
\big{(} \beta_{n q}[x_{n},x_{q}]\otimes x_{k} G'\big{)}\big{)}+\alpha_{ij k}[x_j,x_{k}]\otimes x_iG' +\\&\alpha_{ij k}[x_{k},x_i]\otimes x_j G'\Big{)}=0. \end{align*}
Now,
\begin{align*}
&\sum_{i <j<k,1\leq k \leq d} 
\Big{(}\big{(}\alpha_{ijk} 
\sum_{\overset{1\leq n<q\leq d}{(n,q)\neq (i, j)}}
\big{(}\beta_{n q}[x_{n},x_{q}]\otimes x_{k} G'\big{)}\big{)}+\alpha_{ij k}[x_j,x_{k}]\otimes x_iG' +\\&\alpha_{ij k}[x_{k},x_i]\otimes x_j G'\Big{)}=\sum_{1\leq k \leq d,i<j<k} 
\Big{(}\big{(}\alpha_{ijk} 
\sum_{\overset{1\leq n<q<k\leq d}{(n,q)\neq (i, j)}}
\beta_{n q}[x_{n},x_{q}]\otimes x_{k} G'\big{)}+\\& \big{(}\alpha_{ijk}\sum_{\overset{1\leq n<q\leq d,q\geq k}{(n,q)\neq (i, j)}}
\beta_{n q}[x_{n},x_{q}]\otimes x_{k} G'\big{)}+\big{(}\alpha_{ij k}[x_j,x_{k}]\otimes x_iG' +\alpha_{ij k}[x_{k},x_i]\otimes x_j G'\big{)}\Big{)}.
\end{align*}
Thus
\begin{align*}
&\sum_{1\leq n<q<k \leq d} \Big{(}\alpha_{nqk}[x_n,x_q]\otimes x_k G'+\alpha_{nqk}[x_q,x_k]\otimes x_nG' +\alpha_{nqk}[x_k,x_n]\otimes x_qG'\Big{)}=\\&
\sum_{\overset{1\leq n<q<k \leq d}{(n,q)\neq (i,j),i<j<k}} \Big{(}
\big{(}(\alpha_{nqk}+ \alpha_{i j k}\beta_{nq})[x_n,x_q]\otimes x_k G'+\alpha_{nqk}[x_q,x_k]\otimes x_nG' +\\&\alpha_{nqk}[x_k,x_n]\otimes x_qG'\big{)}+
\sum_{\overset{1\leq n<q<k \leq d}{(n,q)\neq (i,j),j\geq k} }\big{(}\alpha_{nqk}[x_n,x_q]\otimes x_k G'+\alpha_{nqk}[x_q,x_k]\otimes x_nG' \\&+\alpha_{nqk}[x_k,x_n]\otimes x_qG'\big{)}
+
\sum_{1\leq k \leq d,i<j<k} 
\Big{(} \big{(}\alpha_{ijk}\sum_{\overset{1\leq n<q\leq d,q\geq k}{(n,q)\neq (i, j)}}
\beta_{n q}[x_{n},x_{q}]\otimes x_{k} G'\big{)}+\\&\big{(}\alpha_{ij k}[x_j,x_{k}]\otimes x_iG' +\alpha_{ij k}[x_{k},x_i]\otimes x_j G'\big{)}\Big{)}=0.
\end{align*}
For all $ 1\leq n<q<k \leq d,(n,q)\neq (i,j), $  we have $[x_q,x_k]\otimes x_nG',[x_q,x_k]\otimes x_nG',[x_k,x_n]\otimes x_qG'\in A  $  and also if $i<j<k,  $ then $[x_j,x_{k}]\otimes x_iG',[x_{k},x_i]\otimes x_j G'\in A.  $ 
 Since $ A $ is linearly independent, we have $ \alpha_{nqk}=0 $ for all $1\leq n<q<k \leq d.  $
It implies $B  $ is linearly independent. It is easy to see that $ \dim B= \frac{1}{6}d(d-1)(d-2).$ Therefore $|\mathrm{Im}\Psi_2|=p^{\frac{1}{6}d(d-1)(d-2)}.$
\end{proof}

\begin{thm}\label{k5}
Let $ G $ be a $d$-generator special $p$-group of rank $ \frac{1}{2}d(d-1)-1$ and $  d\geq 3. $ Then
\begin{itemize}
\item[$(i)$] $\mathcal{M}(G)  \cong \mathbb{Z}^{\frac{1}{3}d(d-1)(d+1)-d+1-m}_p $ and
 $ \ker \beta\cong \mathbb{Z}^{\frac{1}{6}d(d-1)(d-2)+m}_p, $ where $m=\log_p \dfrac{|\langle w^p\otimes wG'|w\in G\rangle |}{|\mathrm{Im}\Psi_2\cap \langle w^p\otimes wG'|w\in G\rangle |}.$
\item[$(ii)$] $G\cong  \langle y_1,\ldots,y_d|y_t^p=[y_k,y_r]^p=1,[y_k,y_r,y_t]=1,[y_i,y_j]=$ \newline $\prod_{\overset{1\leq n_1<n_2\leq d}{(n_1,n_2)\neq (i, j )}}[y_{n_1},y_{n_2}]^{\alpha_{n_1n_2}},1\leq k,r,t\leq d,0\leq \alpha_{n_1n_2}<p\rangle $ for some $ i,j  $ if and only if $ \mathcal{M}(G)\cong \mathbb{Z}^{(\frac{1}{3}d(d-1)(d+1)-d+1)}_p. $
\end{itemize}

\end{thm}
\begin{proof}
$(i).$
Proposition \ref{lklklk} and Theorem \ref{lkk} imply
\[ |\mathcal{M}(G)|=\dfrac{|G'\otimes G/G'||\mathcal{M}(G/G')| }{|G'||\ker \beta| }~ \text{and}~\ker \beta=\mathrm{Im}\Psi_2 \langle w^p\otimes wG'|w\in G\rangle.\]
Let $\{ x_1, x_2,\ldots, x_d\}$ be a minimal generating set of $G.$ Since $  G/G' $ and $ G' $ are elementary abelian $p$-groups, they can be considered as vector spaces over the  field  $\mathbb{Z}_p.$ Therefore $ G'\otimes G/G' $ can be considered as a vector space with a subspace $ \ker \beta. $   Since by using Proposition \ref{hklkl}, $|\mathrm{Im}\Psi_2 | =p^{\frac{1}{6}d(d-1)(d-2)}, $ we have $ |\ker \beta|=p^{\frac{1}{6}d(d-1)(d-2)+m },$ where $ m=log_p\dfrac{|\langle w^p\otimes wG'|w\in G\rangle |}{|\mathrm{Im}\Psi_2\cap \langle w^p\otimes wG'|w\in G\rangle |}.$ Lemma \ref{11} implies $ |\mathcal{M}(G/G')|= p^{\frac{1}{2}d(d-1)}.$ Thus
\begin{align*}
|\mathcal{M}(G)|&=\dfrac{|G'\otimes G/G'||\mathcal{M}(G/G')| }{|G'||\ker \beta| }=\\&p^{\frac{1}{2}d^2(d-1)-d+\frac{1}{2}d(d-1) -\frac{1}{6}d(d-1)(d-2)- \frac{1}{2}d(d-1)+1-m}=\\&
p^{\frac{1}{2}d^2(d-1) -\frac{1}{6}d(d-1)(d-2)-d+1-m}=\\&
p^{\frac{1}{3}d(d-1)(d+1)-d+1-m}.
\end{align*}
Therefore $ |\mathcal{M}(G)| =  p^{\frac{1}{3}d(d-1)(d+1)-d+1-m}.$ It completes the proof  $ (i). $  \\
$ (ii). $
First, we claim   $G$ is of exponent $p$ if and only if $|\mathcal{M}(G)| = p^{\frac{1}{3}d(d-1)(d+1)-d+1}.$   If $ exp(G)=p, $ then $m=0  $  and so  $ |\mathcal{M}(G)| \leq  p^{\frac{1}{3}d(d-1)(d+1)-d+1},$ by $ (i). $
Conversely, let $ m=0. $ Then 
$\mathrm{Im}\Psi_2\cap \langle w^p\otimes wG'|w\in G\rangle =\langle w^p\otimes wG'|w\in G\rangle.$ By contrary, let $exp(G)\neq p.$ Since $ exp(G)\leq p^2, $ we have $exp(G)=p^2.$ Therefore   $ x_l^p \neq 1$ for some $l$ and  so $x_l^p\otimes x_l G'\neq 0.$   Let $ x_l^p\otimes x_l G'\in \mathrm{Im} \Psi_2.$ 
 By using a same technique used  in the proof of Proposition \ref{hklkl},   one can see that $x_l^p\otimes x_l G'  $ is not generated by elements of $ B.$ Thus 
$x_l^p\otimes x_l G' \notin \mathrm{Im} \Psi_2 $ which is a contradiction.
%
  Therefore $ m =0$ if and only if $G$ is of exponent $p$ and so $|\mathcal{M}(G)| = p^{\frac{1}{3}d(d-1)(d+1)-d+1}$ if and only if $G$ is of exponent $p.$  Note that $ G'=\langle [x_{n_1},x_{n_2}]|1\leq n_1<n_2\leq d\rangle. $ Since $ G' $ is elementary abelian of rank $ \frac{1}{2}d(d-1)-1,$ we have $ [x_i,x_j]=\prod_{\overset{1\leq n_1<n_2\leq d}{(n_1,n_2)\neq (i,j) }}[x_{n_1},x_{n_2}]^{\alpha_{n_1n_2}}, $ where $0\leq  \alpha_{n_1n_2} < p,$ for all $1\leq i<j \leq d.$  \cite[Proposition 5.1]{roz} implies  $ G $ is isomorphic to the quotient of a $ d$-generator free relatively special $p$-group. Thus we get $G\cong  \langle x_1,\ldots,x_d|x_t^p=[x_k,x_r]^p=[x_k,x_r,x_t]=1,1\leq k,r,t\leq d\rangle /\langle [x_i,x_j]^{-1}\prod_{\overset{1\leq n_1<n_2\leq d}{(n_1,n_2)\neq (i,j) }}[x_{n_1},x_{n_2}]^{\alpha_{n_1n_2}}\rangle$ for some $ i,j, $ by using
Lemma \ref{h1}. Since $ G $ is of exponent $ p,$ one can see that  $G\cong  \langle y_1,\ldots,y_d|y_t^p=[y_k,y_r]^p=1,[y_r,y_,y_t]=1,[y_i,y_j]=\prod_{\overset{1\leq n_1<n_2\leq d}{(n_1,n_2)\neq (i,j) }}[y_{n_1},y_{n_2}],1\leq k,r,t\leq d,0\leq  \alpha_{n_1n_2} \leq p\rangle $ for some $ i,j.  $ The proof of case $ (ii) $ is completed.
\end{proof}
\begin{thm}\label{mmmm}
Let $ H $ be a $d$-generator special $p$-group,  $  d\geq 3 $ and $(p\neq 2)$. Then $| H'|=p^{\frac{1}{2}d(d-1)}$ and $ exp(H)=p $ if and only if
$ |\mathcal{M}(H)|=p^{\frac{1}{3}d(d-1)(d+1)}. $
\end{thm}
\begin{proof}
Let $ |\mathcal{M}(H)|=  p^{\frac{1}{3}d(d-1)(d+1)}. $ By a similar way involved in the proof of Theorem \ref{k5}, we can see that if $| H'|<p^{\frac{1}{2}d(d-1)},$ then $| \mathcal{M}(H)|< p^{\frac{1}{3}d(d-1)(d+1)}.$ Thus $|H'|=p^{\frac{1}{2}d(d-1)}.$ The converse holds by  Theorem \ref{h2} $ (ii) $.
\end{proof}
\begin{thm}\label{h6}
Let $ H $ be a  $d$-generator   special $p$-group of rank $ \frac{1}{2}d(d-1)$  of exponent $p,$  $  d\geq 3 $ and $ (p\neq 2). $ Then $ H $ is capable.
\end{thm}
\begin{proof}
Theorems \ref{h2} $ (ii) $ and \ref{k5} $ (ii) $ imply $ |  \mathcal{M}(H)|= p^{\frac{1}{3}d(d-1)(d+1)}$ and
 $ |  \mathcal{M}(H/K)|$  $= p^{\frac{1}{3}d(d-1)(d+1)-d+1} $  for every  central subgroup $  K $ of order $p $ in $ H. $ Since $ |  \mathcal{M}(H)|> | \mathcal{M}(H/K)|,  $ the homomorphism $    \mathcal{M}(H)\rightarrow  \mathcal{M}(H/K) $ cannot be a monomorphism. Now the result follows from Lemma \ref{a}.
\end{proof}

\begin{cor}
Let $ H$ is a $d$-generator relatively free  special $p$-group of rank  $ \frac{1}{2}d(d-1),$  $  d\geq 3 $ and  exponent $ p $ $(p\neq 2)$. Then $ H $ is capable.
\end{cor}

\begin{proof}
The result follows from Lemma \ref{h1} and Theorem \ref{h6}.
\end{proof}
\begin{lem}\label{hhhh}
Let $ G $ be a $d$-generator special $p$-group of rank $ \frac{1}{2}d(d-1)-1$ or $ \frac{1}{2}d(d-1)$ and  $  d\geq 3$ $(p\neq 2)$. Then $\mathrm{Im}\Psi_2\cap \langle w^p\otimes wG'|w  \in G\rangle=0.$
\end{lem}
\begin{proof}
Let $\{ x_1, x_2,\ldots, x_d\}$ be a minimal generating set of $G.$
If  $exp(G)=p,$ then the result easily obtained. 
Let $exp(G)\neq p.$ Since $ exp(G)\leq p^2, $ we have $exp(G)=p^2.$  By contrary, let $\mathrm{Im}\Psi_2\cap \langle w^p\otimes wG'|w  \in G\rangle\neq 0.$  Therefore   $ w^p \neq 0$  and  so $ w^p \otimes w G'\neq 0 $ such that $ w^p\otimes w G'\in \mathrm{Im} \Psi_2 $ and so $w^p\otimes w G'
=\sum_{1\leq  k<k'<r\leq d} \delta_{kk'r}([x_k,x_{k'}]\otimes x_r G'+[x_{k'},x_r]\otimes x_k G'+[x_r,x_k]\otimes x_{k'} G')\in \mathrm{Im} \Psi_2.$ By the proof of Proposition \ref{hklkl} and \cite[Theorem 1.2]{raif},  one can see that $w^p\otimes w G'  $ is not generated by elements of  $B= \{ \Psi_2(x_rG'\otimes x_kG'\otimes x_tG')|1\leq r<k<t \leq d \}.$ It is a contradiction since $\mathrm{Im}\Psi_2\cap \langle w^p\otimes wG'|w  \in G\rangle\neq 0$ and so the result holds.
\end{proof}
\begin{cor}\label{kkk55}
Let $ G $ be a $d$-generator special $p$-group of rank $ \frac{1}{2}d(d-1),$ $  d\geq 3 ,$ $ p\neq 2 $ and $| G^p|=p^t.$ Then
 $\mathcal{M}(G)  \cong \mathbb{Z}^{\frac{1}{3}d(d-1)(d+1)-\frac{1}{2}t(2d-t+1)}_p. $
\end{cor}
\begin{proof}
By  Theorem \ref{h2}, Proposition \ref{hklkl5} and Lemma \ref{hhhh}, we have $ m=log_p  |\langle w^p\otimes wG'|w\in G\rangle |=\frac{1}{2}t(2d-t+1),$ as required.
\end{proof}
\begin{cor}\label{kkk5}
Let $ G $ be a $d$-generator special $p$-group of rank $ \frac{1}{2}d(d-1)-1$ for $ p\neq 2,  d\geq 3$ and $| G^p|=p^t. $ Then
 $\mathcal{M}(G)  \cong \mathbb{Z}^{\frac{1}{3}d(d-1)(d+1)-d+1-\frac{1}{2}t(2d-t+1)}_p. $
\end{cor}
\begin{proof}
By Proposition \ref{hklkl5}, Theorem \ref{k5} and Lemma \ref{hhhh}, we have $ m'=log_p  |\langle w^p\otimes wG'|w\in G\rangle |=\frac{1}{2}t(2d-t+1).$ Hence the result follwos.
\end{proof}

\begin{lem}\label{5f}
Let $ H $ be a $d$-generator   special $p$-group of rank $ \frac{1}{2}d(d-1)$  for $ (p\neq 2),   d\geq 3 $ and $  K $ be a central subgroup  of order $p. $ Then  $ t_1=t-1 $ or $ t_1=t, $ where  $ |H^p|=p^{t}$ and $ |(H/K)^p|=p^{t_1}. $
\end{lem}
\begin{proof}
 We know that $0\leq t,t_1\leq d. $ Since $( H/K)^p=H^pK/K, $ we have $ p^{t_1}=|H^p|/|H^p \cap K| $ and so $ t_1=t-1 $ or $ t_1=t, $ as required.
\end{proof}
Using the notation of Lemma \ref{5f}, we have
\begin{thm}\label{h661l}
Let $ H $ be a  $d$-generator   special $p$-group of rank $ \frac{1}{2}d(d-1)$  for $d\geq 3 $ and $ (p\neq 2). $  Then $ H $ is non-capable if and only if $ H^p\cong  \mathbb{Z}_p.$

\end{thm}
\begin{proof}
Let $ H $ be non-capable and $  K $ be a central subgroup  of order $p $ in $Z^{*}( H). $
Theorem \ref{h6} implies $ exp(H)=p^2.$  Thus the homomorphism $    \mathcal{M}(H)\rightarrow  \mathcal{M}(H/K) $ is  a monomorphism, by \cite[Theorem 4.2]{3}. Hence $|  \mathcal{M}(H)|= |  \mathcal{M}(H/K)|p^{-1}, $ by using \cite[Lemma 4.1]{3}.
Corollaries \ref{kkk55} and \ref{kkk5} imply \[ |  \mathcal{M}(H)|= p^{\frac{1}{3}d(d-1)(d+1)-\frac{1}{2}t(2d-t+1)}\] and \[
  |  \mathcal{M}(H/K)|= p^{\frac{1}{3}d(d-1)(d+1)-d+1-\frac{1}{2}t_1(2d-t_1+1)} \]   such that $ |(H/K)^p|=p^{t_1} $ and $ |H^p|=p^{t}.$ Thus
  \begin{align*}
   p^{\frac{1}{3}d(d-1)(d+1)-\frac{1}{2}t(2d-t+1)}&=|  \mathcal{M}(H)|= |  \mathcal{M}(H/K)|p^{-1}\\&=p^{\frac{1}{3}d(d-1)(d+1)-d-\frac{1}{2}t_1(2d-t_1+1)}.
 \end{align*}
 Therefore $t(2d-t+1)-t_1(2d-t_1+1)=2d.  $
 By Lemma \ref{5f}, $ t_1=t $ or $ t_1=t-1. $ The case  $ t_1=t $ cannot be occur since $ d>0. $ Thus $ t_1=t-1 $
 and so $ t(2d-t+1)-(t-1)(2d-t+2)=2d $ and so $ t=1. $  It follwos that $ H^p\cong  \mathbb{Z}_p.$
 Conversely, let $ H^p\cong  \mathbb{Z}_p.$ By Corollaries \ref{kkk55} and \ref{kkk5}, $ |  \mathcal{M}(H)|= p^{\frac{1}{3}d(d-1)(d+1)-d}$ and
 $ |  \mathcal{M}(H/H^p)|= p^{\frac{1}{3}d(d-1)(d+1)-d+1}. $  By using \cite[Lemma 4.1]{3}, $ |  \mathcal{M}(H)|=|  \mathcal{M}(H/H^p)|p^{-1}.$ Now \cite[Theorem 4.2]{3} implies $ H $ is non-capable. The proof is completed.
\end{proof}
\begin{cor}\label{h66}
Let $ H $ be a  $d$-generator  special $p$-group of rank $ \frac{1}{2}d(d-1),$  $d\geq 3 $ and $ (p\neq 2).$  Then $ H $ is non-capable if and only if  $H\cong \langle x_1,\ldots,x_d|x_k^{p}=[x_i,x_j]^p=1,[x_1,x_2]=x_1^p,[x_i,x_j,x_t]=1,1\leq i,j,k,t\leq d, k\neq 1\rangle.$
\end{cor}
\begin{proof}
Let $ H $ be non-capable. By Theorem \ref{h661l}, we have  $ H^p\cong  \mathbb{Z}_p.$
We can choose a generating set $ \{ x_1Z(H),\ldots,x_dZ(H)\} $ for $ H/Z(H)$ such that $ [x_i,x_j] $ is non-trivial for some $ i\neq j. $
It is clear   to see that $ \{[x_i,x_j]|1\leq i<j\leq d\} $ is a generating set of $ H'.$ Since $Z(H)= H'\cong \mathbb{Z}^{(\frac{1}{2}d(d-1))}_p$ and $ H^p\cong  \mathbb{Z}_p,$  we have
$H\cong \langle x_1,\ldots,x_d|x_k^{p}=[x_i,x_j]^p=1,[x_1,x_2]=x_1^p,[x_i,x_j,x_t]=1,1\leq i,j,k,t\leq d,k\neq 1\rangle.$  The converse is clear by using Theorem \ref{h661l}.
\end{proof}
The following corollary shows the converse of Theorem \ref{h661l} is also true.
\begin{cor}
Let $ H $ be a  $d$-generator   special $p$-group of rank $ \frac{1}{2}d(d-1),$ $d\geq 3 $ and $ (p\neq 2). $ Then $ H $ is capable if and only if $| H^p|\neq p.$
\end{cor}

\section{The tensor square and exterior square of $ d$-generator special $p$-groups of rank $ \frac{1}{2}d(d-1)$ and $ \frac{1}{2}d(d-1)-1$ }
Let $ H $ be a $ d$-generator special $p$-group of rank $ \frac{1}{2}d(d-1)$ or $ \frac{1}{2}d(d-1)-1.$
This section is devoted to give an upper bound for the tensor square, the exterior, and the third homology group $ \pi_3(SK(H,1)) $ of suspension an Eilenberg- MacLane space $ K(H,1) $ which is isomorphic to $J_2(H) $. When $ H $ is  of exponent $ p,$  we characterize the explicit structure of $ H\wedge H, H\otimes H$ and $ J_2(H). $
\begin{thm}\label{kjkj}
  Let $ G $ be a  finite  special $p$-group and $ exp(G)=p^2 $  $(p\neq 2)$. If  $ G\wedge G $ is elementary, then $ G $ is non-capable.
\end{thm}
\begin{proof}
 Let
 $G= \langle x_1,\ldots,x_d\rangle. $ We claim that $ 1 \neq  Z^{\wedge}(G). $ Since $ exp(G)=p^2, $ there exists $ i, $  $ 1\leq i\leq d $ such that $ x_i^p\neq 1. $ For all $ 1\leq j\leq d, $ we have
\begin{align*}
&x_i^{p}\wedge  x_j=\prod_{t=p-1}^0 (^{x_i^t}(x_i\wedge x_j))=\prod_{t=p-1}^0 (x_i\wedge [x_i, x_j^t]x_j)\\&=\prod_{t=p-1}^1 (x_i\wedge [x_i, x_j])^t\prod_{t=p-1}^0 (x_i\wedge x_j).\end{align*}
Clearly, $\prod_{t=p-1}^1 (x_i\wedge [x_i,x_j])^t=(x_i\wedge [x_i, x_j])^{\frac{1}{2}p(p-1)} =1_{G\wedge G}. $
Since $ G\wedge G $ is elementary, we have $x_i^{p}\wedge  x_j=\prod_{t=p-1}^0 (x_i\wedge x_j)=(x_i\wedge x_j)^{p}=1_{G\wedge G},$
for all $ 1\leq j\leq d. $ Therefore $x_i^{p}\wedge  x_j=1_{G\wedge G},$
for all $ 1\leq j\leq d. $ Hence $ 1\neq x_i^p\in Z^{\wedge}(G) $ and so  $ G $ is non-capable, as required.
\end{proof}
\begin{lem}\label{lle}
Let $ G $ be a $p$-group of class two. Then
$ G\wedge G $ is abelian.
\end{lem}
\begin{proof}
Since  $[x_l\wedge x_k, x_m\wedge x_n]=[x_l, x_k]\wedge [x_m, x_n]=^{[x_l, x_k]}(x_m\wedge  x_n)(x_m\wedge  x_n)^{-1}=1,$ for all $x_l, x_k,x_m,  x_n \in G,$ we have $ (G\wedge G)'=1. $  Thus $ G\wedge G$ is abelian.
\end{proof}
\begin{thm}\label{sp}
 Let $ H $ be a   $ d$-generator  special $p$-group   of rank $ \frac{1}{2}d(d-1),$ $d\geq 3 $ and $ (p\neq 2). $ Then
  $ H\wedge H$ is abelian and $ |H\wedge H| \leq p^{\frac{1}{6}d(d-1)(2d+5)}.$ Moreover,      $ exp(H)=p $ if and only if $H\wedge H  \cong \mathbb{Z}^{(\frac{1}{6}d(d-1)(2d+5))}_p. $
\end{thm}
\begin{proof}
By Lemma \ref{lle}, $ H\wedge H$ is abelian.
Now using Lemma \ref{j1}, we have  $|H\wedge H|=| \mathcal{M}(H)||H'|.  $ From Theorem \ref{h2} $ (i), $ we have $ |H\wedge H| \leq p^{\frac{1}{2}d(d-1)+\frac{1}{3}d(d-1)(d+1)}=p^{\frac{1}{6}d(d-1)(2d+5)}.$ Let $ exp(H)=p. $ Then    Theorem \ref{h2} $ (ii) $    implies   $H\wedge H\cong  \mathcal{M}(H)\oplus H'\cong \mathbb{Z}^{(\frac{1}{6}d(d-1)(2d+5))}_p. $ Conversely, let $ H\wedge H \cong \mathbb{Z}^{(\frac{1}{6}d(d-1)(2d+5))}_p.$ Thus $ |\mathcal{M}(H)|=p^{(\frac{1}{6}(d-1)(2d+5)-\frac{1}{2}d(d-1))} = p^{\frac{1}{3}d(d-1)(d+1)}.$ By Theorem \ref{h2} $ (ii),$ we obtain $ exp(H)=p. $ The proof is completed.
\end{proof}

\begin{cor}\label{jlklll}
  Let $H $ be a    $d$-generator special $p$-group of rank $ \frac{1}{2}d(d-1),$ $d\geq 3 $ and $ (p\neq 2). $ Then $ |H\otimes H|\leq p^{\frac{1}{3}d(d^2+3d-1)}.$ Moreover, $ exp(H)=p $ if and only if $H\otimes H\cong \mathcal{M}(H)\oplus H'\oplus \bigtriangledown (H) \cong \mathbb{Z}^{(\frac{1}{3}d(d^2+3d-1))}_p. $
\end{cor}
\begin{proof}
By \cite[Lemma 1.2 $(i)$ ]{som}, $ \bigtriangledown (H)\cong \mathbb{Z}^{(\frac{1}{2}d(d+1))}_p. $
Using \cite[Proposition 2.2 $(iii)$ ]{som}, we have  $H\otimes H \cong (H\wedge H) \oplus \bigtriangledown (H).$  Theorem \ref{sp} implies  $ |H\otimes H|\leq p^{\frac{1}{6}d(d-1)(2d+5)+\frac{1}{2}d(d+1)}=p^{\frac{1}{3}d(d^2+3d-1)}.$ Let $ exp(H)=p. $ By Theorem \ref{sp},     $H\otimes H \cong \mathbb{Z}^{(\frac{1}{3}d(d^2+3d-1))}_p. $ Conversely, let now $H\otimes H \cong \mathbb{Z}^{(\frac{1}{3}d(d^2+3d-1))}_p. $ Then $ H\wedge H \cong (H\otimes H)/\bigtriangledown (H)\cong \mathbb{Z}^{(\frac{1}{6}d(d-1)(2d+5))}_p. $ The result now obtained from Theorem \ref{sp}. So $ exp(H)=p.$
\end{proof}
\begin{cor}\label{jlkll}
Let $H $ be a   $d$-generator special $p$-group of rank $ \frac{1}{2}d(d-1),$ $d\geq 3 $ and $ (p\neq 2). $   Then
$ |J_2(H)|\leq p^{\frac{1}{6}d(d+1)(2d-1)}.$  Moreover,    $ exp(H)=p $ if and only if $J_2(H)\cong \mathcal{M}(H)\oplus \bigtriangledown (H) \cong \mathbb{Z}^{(\frac{1}{6}d(d+1)(2d-1))}_p.$
\end{cor}
\begin{proof}
By \cite[Lemma 1.2 $(i)$ ]{som}, $ \bigtriangledown (H)\cong \mathbb{Z}^{(\frac{1}{2}d(d+1))}_p. $
The result follows from Theorem \ref{h2} and \cite[Corollary 1.4]{som}.
\end{proof}
\begin{cor}
Let $ H $ be a $d$-generator relatively free  special $p$-group of exponent $ p, $  $d\geq 3 $ and $ (p\neq 2). $ Then
\begin{itemize}
\item[$(i)$]$ \mathcal{M}(H)\cong \mathbb{Z}^{(\frac{1}{3}d(d-1)(d+1))}_p.$
\item[$(ii)$]
$H\wedge H  \cong \mathbb{Z}^{(\frac{1}{6}d(d-1)(2d+5))}_p. $
\item[$(iii)$] $H\otimes H  \cong \mathbb{Z}^{(\frac{1}{3}d(d^2+3d-1))}_p. $
\item[$(iv)$]
$ J_2(H)\cong  \mathbb{Z}^{(\frac{1}{6}d(d+1)(2d-1))}_p.$
\end{itemize}
\end{cor}
\begin{proof}
The result follows from Lemma \ref{h1}, Theorem \ref{sp}, Corollaries \ref{jlklll} and \ref{jlkll}.
\end{proof}

\begin{thm}\label{spp}
 Let $ H $ be a   $ d$-generator  special $p$-group  of rank $\frac{1}{2}d(d-1),$   $d\geq 3, $  $(p\neq 2)$ and $ |H^p|=p^t. $ Then
 \begin{itemize}
\item[$(i)$]$ \mathcal{M}(H)\cong \mathbb{Z}^{(\frac{1}{3}d(d-1)(d+1)-\frac{1}{2}t(2d-t+1))}_p.$
\item[$(ii)$]
$|H\wedge H | = p^{\frac{1}{6}d(d-1)(2d+5)-\frac{1}{2}t(2d-t+1)}. $
\item[$(iii)$] $|H\otimes H| = p^{\frac{1}{3}d(d^2+3d-1)-\frac{1}{2}t(2d-t+1)}. $
\item[$(iv)$]
$   J_2(H) \cong \mathbb{Z}^{(\frac{1}{6}d(d+1)(2d-1)-\frac{1}{2}t(2d-t+1))}_p.$
\end{itemize}
\end{thm}
\begin{proof}
 By using Corollary \ref{kkk55},  we have $\mathcal{M}(H)  \cong \mathbb{Z}^{(\frac{1}{3}d(d-1)(d+1)-\frac{1}{2}t(2d-t+1))}_p. $ The rest of proof is obtained by a  similar way used in the proof of Theorem \ref{sp}, Corollaries \ref{jlklll} and \ref{jlkll}.
\end{proof}

\begin{thm}\label{spphj}
 Let $ H $ be a   $ d$-generator  special $p$-group  of rank $\frac{1}{2}d(d-1)$   minimally generated by the set $\{ x_1, x_2,\ldots, x_d\}$ such that $H^p=\langle  x_1^{p}, x_2^{p},\ldots, x_t^{p}\rangle,  $ $d\geq 3, $  $(p\neq 2)$ and $ |H^p|=p^t.$ Then
 \begin{align*} H\wedge H&= \langle x_i\wedge x_j|1\leq i \leq d,t+1\leq j\leq d, i<j \rangle \oplus (\langle x_i\wedge x_s|1\leq i<s \leq t \rangle \ker \kappa')\cong \\& \mathbb{Z}_{p^{2}}^{(\frac{1}{2}t(t-1))} \oplus \mathbb{Z}_{p}^{(\frac{1}{3}(d-1)d(d+1)+\frac{1}{2}(d-1)d-td-\frac{1}{2}t(t-1))},  \end{align*} where 
 $\langle x_i\wedge x_s|1\leq i<s \leq t \rangle \cong \mathbb{Z}_{p^{2}}^{(\frac{1}{2}t(t-1))},  $   $\langle x_i\wedge x_j|1\leq i \leq d,t+1\leq j\leq d, i<j \rangle \cong \mathbb{Z}_{p}^{(\frac{1}{2}(d-1)d-\frac{1}{2}t(t-1))}  $ and $ \ker \kappa'=\langle [x_l,x_r]\wedge x_k|1\leq l<r<\leq d,1\leq k\leq r  \rangle\cong \mathcal{M}(H).$
 \end{thm}
\begin{proof}
It is clear   to see that $ \{[x_i,x_j]|1\leq i<j\leq d\} $ is a generating set of $ H'.$
 Clearly, by using \cite[Proposition 2.4]{1b}, we have $ H\wedge H=\langle x_n\wedge x_q|1\leq n<q \leq d \rangle \langle [x_l,x_r]\wedge x_k|1\leq l<r\leq d,1\leq k\leq d  \rangle.$ Lemma \ref{lclass} implies $ \langle [x_l,x_r]\wedge x_k|1\leq l<r\leq d,1\leq k\leq d  \rangle =\langle [x_l,x_r]\wedge x_k|1\leq l<r\leq d,1\leq k\leq r  \rangle. $ Therefore
 $ H\wedge H=\langle x_i\wedge x_s|1\leq i<s \leq t \rangle \langle x_i\wedge x_j|1\leq i \leq d,t+1\leq j\leq d, i<j \rangle 
 \langle [x_l,x_r]\wedge x_k|1\leq l<r\leq d,1\leq k\leq r  \rangle.$
 Since 
the number of $A= \{[x_l,x_r]\wedge x_k|1\leq l<r\leq d,1\leq k\leq r \} $ is equal to the basic commutators of weight $3$ on $d$
generators, $|A| = \frac{1}{3}d(d-1)(d+1),$  by Theorem \ref{jkhhhhh}.  By a similar way used in the proof of Theorem \ref{kjkj}, we have $ x_i^p\wedge x_s
=(x_i\wedge x_s)^p=x_i\wedge x_s^p $  for all $ 1\leq i<s\leq d. $ 
If $ 1\leq i<s\leq t, $ then $ (x_i\wedge x_s)^{p^2}=1_{H\wedge H}. $
If $1\leq i\leq d, t+1\leq s \leq d, i<s,$ then  $ (x_i\wedge x_s)^{p}=1_{H\wedge H}. $ Thus $\langle x_i\wedge x_s|1\leq i<s \leq t \rangle \cong \mathbb{Z}_{p^{2}}^{(\frac{1}{2}t(t-1))}  $ and so  $\langle x_i\wedge x_j|1\leq i \leq d,t+1\leq j\leq d, i<j \rangle \cong \mathbb{Z}_{p}^{(\frac{1}{2}(d-1)d-\frac{1}{2}t(t-1))}.   $
Since $ x_i^p\wedge x_s=x_i\wedge x_s^p=(x_i\wedge  x_s)^p$ for all 
$ 1\leq i<s\leq t, $ the  number of such elements is $ \frac{1}{2}t(t-1). $ Putting
$ K_1= \langle x_i\wedge x_s|1\leq i<s \leq t \rangle \langle x_i\wedge x_j|1\leq i \leq d,t+1\leq j\leq d, i<j \rangle. $
By using Theorem \ref{spp}, we have
\begin{align*}
&|\langle [x_l,x_r]\wedge x_k|1\leq l<r\leq d,1\leq k\leq r  \rangle |=\\&|H\wedge H|
|\langle [x_l,x_r]\wedge x_k|1\leq l<r\leq d,1\leq k\leq r  \rangle \cap  K_1|/|K_1|=\\&p^{\frac{1}{6}d(d-1)(2d+5)-\frac{1}{2}t(2d-t+1)}p^{\frac{1}{2}t(t-1)}/p^{t(t-1)}p^{\frac{1}{2}d(d-1)-\frac{1}{2}t(t-1)}=\\&
p^{\frac{1}{3}d(d-1)(d+1)-\frac{1}{2}t(2d-t+1))}.
\end{align*}
The order of non-trivial elements $[x_l,x_r]\wedge x_k  $ for all $ 1\leq l<r\leq d,$  $1< k\leq r $ is exactly $p.$ Thus
$\langle [x_l,x_r]\wedge x_k|1\leq l<r\leq d,1\leq k\leq r  \rangle \cong \mathbb{Z}_{p}^{(\frac{1}{3}(d-1)d(d+1)-\frac{1}{2}t(2d-t+1)}\cong \mathcal{M}(H),$ by Theorem \ref{spp}. 
%
%
 Therefore,
\begin{align*}
 &H\wedge H= \langle x_i\wedge x_j|1\leq i \leq d,t+1\leq j\leq d, i<j \rangle \oplus \\& (\langle x_i\wedge x_s|1\leq i<s \leq t \rangle \langle [x_l,x_r]\wedge x_k|1\leq l<r\leq d,1\leq k\leq r  \rangle)\cong \\& \mathbb{Z}_{p^{2}}^{(\frac{1}{2}t(t-1))} \oplus \mathbb{Z}_{p}^{(\frac{1}{3}(d-1)d(d+1)+\frac{1}{2}(d-1)d-td-\frac{1}{2}t(t-1))}.  \end{align*}
 The result follows.
\end{proof}
\begin{cor}\label{spphjk}
 Let $ H $ be a   $ d$-generator  special $p$-group  of rank $\frac{1}{2}d(d-1),$   $d\geq 3, $  $(p\neq 2)$ and $ |H^p|=p^t.$ Then
 \[H\otimes  H \cong  \mathbb{Z}_{p^{2}}^{(\frac{1}{2}t(t-1))}\oplus \mathbb{Z}_{p}^{(\frac{1}{3}(d-1)d(d+1)+d^2-td-\frac{1}{2}t(t-1))}.\]
 \end{cor}
\begin{thm}\label{spp}
 Let $ H $ be a   $ d$-generator  special $p$-group  of rank $\frac{1}{2}d(d-1)-1,$   $d\geq 3, $   $(p\neq 2)$ and $ |H^p|=p^t. $ Then
 \begin{itemize}
\item[$(i)$]$ \mathcal{M}(H)\cong \mathbb{Z}^{(\frac{1}{3}(d-1)(d^2+d-3)-\frac{1}{2}t(2d-t+1))}_p.$
\item[$(ii)$]
$|H\wedge H | = p^{\frac{1}{6}(d-1)(2d^2+5d-6)-\frac{1}{2}t(2d-t+1)-1}. $
\item[$(iii)$] $|H\otimes H| = p^{\frac{1}{3}d(d^2+3d-4)-\frac{1}{2}t(2d-t+1)+1}. $
\item[$(iv)$]
$   J_2(H) \cong \mathbb{Z}^{(\frac{1}{3}(d+1)(2d-3)(d+2)-\frac{1}{2}t(2d-t+1)+2)}_p.$
\end{itemize}
\end{thm}
\begin{proof}
Using Lemma \ref{j1}, we have   $|H\wedge H|=|
 \mathcal{M}(H)||H'|.  $ From Corollary \ref{kkk5},  we have $ |H\wedge H|= p^{\frac{1}{2}d(d-1)+\frac{1}{3}d(d-1)(d+1)-d+1-\frac{1}{2}t(2d-t+1)}=p^{\frac{1}{6}(d-1)(2d^2+5d-6)-\frac{1}{2}t(2d-t+1)}.$ By using \cite[Lemma 1.2 $(i)$ ]{som}, we have $ \bigtriangledown (H)\cong \mathbb{Z}^{(\frac{1}{2}d(d+1))}_p. $
Now, \cite[Proposition 2.2 $(iii)$ ]{som} implies  $|H\otimes H |= |H\wedge H| | \bigtriangledown (H)|= p^{\frac{1}{3}(d^3+3d^2-4d)-\frac{1}{2}t(2d-t+1)+1}.$
By invoking \cite[Corollary 1.4]{som} and Corollary \ref{kkk5}, we have \[ J_2(H) \cong \mathbb{Z}^{(\frac{1}{3}(d+1)(2d-3)(d+2)-\frac{1}{2}t(2d-t+1)+2)}_p.\] The proof is completed.
\end{proof}
\begin{thm}
 Let $ H $ be a   $ d$-generator  special $p$-group  of rank $\frac{1}{2}d(d-1)-1,$ $d\geq 3, $ and  $(p\neq 2).$ If
 $ exp(H)=p, $  then
  \begin{align*} &H\wedge H\cong \mathbb{Z}^{(\frac{1}{6}(d-1)(2d^2+5d-6)+1))}_p \\
 &H\otimes H\cong  \mathbb{Z}^{(\frac{1}{3}d(d^2+3d-4)+1)}_p \\
 & J_2(H)\cong  \mathbb{Z}^{((\frac{1}{3}(d+1)(2d-3)(d+2)+2)}_p.
 \end{align*}
 \end{thm}
 \begin{proof}
 Since $ exp(H)=p, $ the result holds by   Theorem \ref{spp}.
 \end{proof}


\begin{thebibliography}{99}%
{\small
\bibitem{1b}
M. R. Bacon and L. C. Kappe, On capable $p$-groups of nilpotency class two, Illinois J. Math. 47 (2003), no. 1-2, 49-62.
\bibitem{3}
F. R. Beyl, U. Felgner, and P. Schmid, On groups occurring as center factor groups, J. Algebra 61 (1979), no. 1, 161-177.
\bibitem{black burn}
N. Blackburn and L. Evens,  Schur multipliers of $p$-groups, J. Reine Angew. Math. 309 (1979), 100-113.
\bibitem{som}
R. D. Blyth, F. Fumagalli and M. Morigi, Some structural results on the non-abelian tensor square of groups. J. Group Theory 13 (2010), no. 1, 83-94.
\bibitem{br1}
R. Brown, D. L. Johnson, and E. F. Robertson, Some computations of non-abelian tensor products of groups, J. Algebra 111 (1987) 177-202.
\bibitem{br2}
R. Brown, J.-L. Loday, Van Kampen theorems for diagrams of spaces, Topology 26 (1987), 311-335.
\bibitem{roz}
R. Cortini, On special $p$-groups, Boll. Unione Mat. Ital. Sez. B Artic. Ric. Mat. (8) 1 (1998), no. 3, 677-689.
\bibitem{ele}
G. Ellis and J. Wiegold, A bound on the Schur multiplier of a prime-power group. Bull. Austral. Math. Soc. 60 (1999), no. 2, 191-196.
\bibitem{111}
G. Ellis, $q$-crossed On the relation between  upper central quotient and lower central series of a group, A.M.S. Soc. 353 (2001) 4219-4234.
\bibitem{el}
G. Ellis, A bound for the derived and Frattini subgroups of a prime-power, Proc. Amer. Math. Soc., 126 No. 9 (1998) 2513-2523.
\bibitem{m}
 M. Hall, The Theory of Groups, MacMillan Company, NewYork, 1959.
\bibitem{Hu}
B. Huppert and N. Blackburn, Finite Groups II (Springer, Berlin, 1982).
\bibitem{kar}
G. Karpilovsky, The Schur Multiplier, London Math. Soc. Monogr., New Ser., vol. 2, 1987. 

\bibitem{nif}
P. Niroomand and F. Johari, A further investigation on the order of the Schur multiplier of $p$-groups
\bibitem{raif}
 P. Rai, On the Schur multiplier of special $ p$-groups,  J. Pure Appl. Algebra, available at https://doi.org/10.1016/j.jpaa.2017.04.004, in press.












}
\end{thebibliography}
\end{document}